\documentclass[12pt]{amsart}
\usepackage{amsmath}
\usepackage{amssymb}
\usepackage{amsthm}
\usepackage{epsfig,graphicx}

\newcommand{\ignore}[1]{\relax}

\newcommand{\supp}{\operatorname{supp}}

\newcommand{\R}{\mathbb R}
\newcommand{\Z}{\mathbb Z}

\newcommand{\lk}{\operatorname{lk}}

\newcommand{\p}{\mathbb P}

\newtheorem{thm}{Theorem}

\newtheorem{lem}{Lemma}[section]
\newtheorem{claim}{Claim}[section]
\newtheorem{cor}[lem]{Corollary}
\newtheorem{thma}[lem]{Theorem}
\newtheorem{theorem}{Theorem}
\newtheorem{corollary}[lem]{Corollary}

\newtheorem{coro}[lem]{Corollary}
\newtheorem{prop}[lem]{Proposition}

\theoremstyle{definition}
\newtheorem{condition}[claim]{Condition}
\newtheorem{defn}[lem]{Definition}

\newtheorem{problem}[claim]{Problem}

\theoremstyle{remark}
\newtheorem{rmk}[lem]{Remark}
\newtheorem{rem}[lem]{Remark}

\renewcommand{\setminus}{\smallsetminus}

\newcommand{\conj}{\operatorname{conj}}

\newcommand{\dd}{\partial}

\newcommand{\cp}{
{\mathbb P}}
\newcommand{\rp}{{\mathbb R}{\mathbb P}}

\newcommand{\pp}{{\mathbb P}}

\newcommand{\Log}{\operatorname{Log}}

\renewcommand{\setminus}{\smallsetminus}

\ignore{

\newtheorem{condition}[theorem]{Condition}

\newenvironment{proof}[1][Proof]{\noindent\textbf{#1.} }{\ \rule{0.5em}{0.5em}}
}

\newcommand{\chord}{\operatorname{Chord}}
\newcommand{\tang}{\operatorname{Tangent}}

\author{Grigory Mikhalkin and Stepan Orevkov}
\title{Rigid isotopy of maximally writhed links}
\address{Universit\'e de Gen\`eve,  Math\'ematiques, Battelle Villa, 1227 Carouge, Suisse}
\address{Steklov Mathematical Institute,
ul. Gubkina 8, 119991, Moscow, Russia and
IMT, Universit\'e Paul Sabatier, 118 route de Narbonne,
31062, Toulouse, France
}
\begin{document}
\begin{abstract}
This is a sequel to the paper \cite{MO-mw}
which identified maximally writhed algebraic
links in $\rp^3$ and classified them topologically.
In this paper we prove that
all maximally writhed
links of the same topological type are rigidly
isotopic, i.e. one can be deformed into another
with a family of smooth real algebraic links
of the same degree.
\end{abstract}
\maketitle

\footnote{Research is supported in part by the grants 178828, 182111
and the NCCR SwissMAP
project of the Swiss National Science Foundation (G.M.), and by RFBR grant No
17-01-00592a (S.O.).}

\section{Introduction}
A real algebraic curve $A\subset\pp^3$ is a (complex)
one-dimensional subvariety invariant with
respect to the involution of complex conjugation
$(z_0:z_1:z_2:z_3)\mapsto
(\bar z_0:\bar z_1:\bar z_2:\bar z_3)$.
We denote with $\R A$ the fixed point locus of $\conj|_{A}$,
note that $\R A=A\cap\rp^3$.
Following the classical terminology by a {\em real branch}
of $A$ we mean the image $\nu(K)$ for a connected component
$K$ of $\R\tilde A$, where $\nu: \tilde A\to A$ is a normalization.

In the case of nonsingular 
$A$
we call $L=\R A$ the 
{\em real algebraic link}, cf. \cite{MO-mw},
if every component of the normalization of $A$
has non-empty real part, and if $L$ is not contained in any proper projective
subspace of $\rp^3$.
We say that $L$ is irreducible if $A$ is irreducible.
In this paper real algebraic links as
well as algebraic curves are
assumed to be irreducible.

The degree $d\in\Z$ of a real algebraic
link $L\subset\rp^3$ 
is a positive number such that
$[A]=d[\pp^1]\in H_2(\pp^3)=\Z$,
the genus of $L$ is the genus of $A$.
Given a point $p\in\rp^3\setminus\R A$ we
denote with 
$
\pi_p:\p^3\setminus\{p\}\to\p^2
$
the linear projection from $p$ so that
$\pi_p(A)\subset\pp^2$ is a 
plane curve of the same degree $d$.
If $p$ is chosen generically then $\pi_p(A)$
is a nodal curve.
The set 
\[
(\pi_p(A)\cap\rp^2)\setminus\pi_p(\R A)
\]
is not always empty. It is a finite set consisting
of {\em solitary nodes} of $\R\pi_p(A)=
\pi_p(A)\cap\rp^2$, i.e. the points of real
intersection of pairs of complex conjugate local branches.
It was shown in \cite{Viro-en} that each node
of $\R\pi_p(A)$ may be prescribed a sign $\pm 1$ or $0$
so that the sum $w(L)$
of the signs of all nodes does not
depend on a choice of $p$ and
is an invariant of $L$.

Non-solitary nodes of $\R\pi_p(A)$ are
intersection points of pairs of real local branches of $A$.
If these local branches belong to different
real branches of $A$ then the sign of
the corresponding node is defined as zero.
Otherwise, the definition of sign at 
a non-solitary node
agrees
with the convention used in Knot Theory
for the definition of the {\em writhe} of 
the knot diagram, 
see Figure \ref{signs} (we choose any
orientations of the two local branches
that can agree with an orientation
of the real branch containing them).
\begin{figure}[h]
\includegraphics[height=18mm]{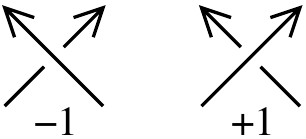}
\caption{Writhe signs for a crossing
point of two real branches on the knot diagram.
\label{signs}}
\end{figure}

In the case of solitary nodes
the signs were introduced in \cite{Viro-en}.
Accordingly, $w(L)$
is called the {\em encomplexed writhe}
(as in \cite{Viro-en}), or the {\em Viro invariant}
of $L\subset\rp^3$.
Note that the number of nodes
of $\pi_p(A)$ is $N_d-g$,
where $N_d=\frac{(d-1)(d-2)}{2}$.
Thus we have the straightforward upper bound
\begin{equation}\label{eq-ew}
|w(L)|\le N_{d}-g\le N_d
\end{equation}
for any real algebraic link $L\subset\rp^3$
of degree $d$. Accordingly, an (irreducible)
real algebraic
link of degree $d$ with $|w(L)|=N_d$ is called
a {\em maximally writhed knot} or
an {\em $MW$-knot} (since this properties implies
that $g=0$, the curve $A$ is rational and thus
$L$ is connected).

The following extremal property of 
\eqref{eq-ew} was shown in \cite{MO-short}.
\begin{thm}[\cite{MO-short}]
If $L,L'\subset\rp^3$ are $MW$-knots
of the same degree then the pairs
$(\rp^3,L)$ and $(\rp^3,L')$ are diffeomorphic.
\end{thm}
In other words, the topological type of an
$MW$-knot of a given degree is unique.

Consider now the case of an arbitrary $g$.
By Harnack's inequality 
the number of real branches of $A$ is at most
$g+1$. The (irreducible) real curves $A$ with $g+1$
real branches
are known as $M$-curves.
In this case the complement
$A\setminus\R A$ consists
of two components bounding $\R A$.
The orientations of two component of
$A\setminus\R A$ yield a pair of opposite
orientations of $L$. Through these orientations,
an orientation of one real branch of $A$ determines
the orientation of all other real branches,
and we define the sign $\sigma(q)=\pm 1$
also in the case if a node $q\in\R\pi_p(A)$
corresponds
to a crossing point
of different real branches of $A$
according to Figure \ref{signs}.
We set
\begin{equation}\label{eq-wl}
w_\lambda(L)=w(L)+
\sum\limits_q\sigma(q)=w(L)+
2\sum\limits_{K\neq K'}\lk(K,K'),
\end{equation}
see \cite{MO-mw}. Here the first sum is taken over
the nodes of $q$ corresponding to intersection
points of different real branches of $A$ while the
second sum is taken over pairs of distinct real branches
$K,K'$.

Clearly, if $|w_\lambda(L)| = N_d$ 
then
$L$ consists of a single real branch
as we may choose the point $p$ so that 
one of the nodes corresponds to a crossing
point of two different branches of $A$.
In the same time we also have the inequality
\[
|w_\lambda(L)|\le N_d-g
\]
which has a chance to be sharp
even for a multicomponent $L$.
\begin{defn}
A {\em maximally writhed} real algebraic link
(or an $MW_\lambda$-link) $L\subset\rp^3$
is a real algebraic link in $g+1$ real branches,
$g\ge 0$,
such that the irreducible real algebraic curve
$A$ with $\R A=L$ is of genus $g$ and that
$|w_\lambda(L)|\le N_d-g$.
\end{defn}
Note that $N_d>g$ as, otherwise $A$ must be planar.
Therefore, $w_\lambda(L)\neq 0$ if
$L$ is an $MW_\lambda$-link.
We refer to the sign of $w_\lambda(L)$ as
the {\em chirality} of an $MW_\lambda$-link $L$. 

Let $\alpha=(a_0,\dots,a_g)$ be a partition
of the number $d-2$ into $l=g+1$ positive integer numbers.
This means that $a_0\ge\dots\ge a_g$ are positive integer 
numbers such that $\sum\limits_{j=0}^ga_j=d-2$.
In \cite{MO-mw} to each $\alpha$ we
have associated 
a (topological) link $W_g(\alpha)$
in $g+1$ components.
Namely, $W_g(\alpha)$ is a
link that sits on the boundary
(the union of $g+1$ hyperboloids) of 
a regular neighborhood of the Hopf link $H$
in $g+1$ component.
Each of the $g+1$ hyperboloids contains
a single component of $W_g(\alpha)$.
Furthermore, the linking number
of the $j$-th component $K_j$ (enhanced
with some orientation)
and the $j$th component of the Hopf link is $a_j+2$
while the linking number of $K_j$ with any other
component of $H$ is $a_j$.
\begin{thm}[\cite{MO-mw}]\label{thm-mwl}
For every $MW_\lambda$-link $L$ of degree $d$
there exists a partition $\alpha=(a_0,\dots,a_g)$ such that 
$(\rp^3,L)$ and $(\rp^3,W_g(\alpha))$
are diffeomorphic.

Conversely,
for every $\alpha$ there exists an $MW_\lambda$-link
$L$ of degree $d$ such that the pairs
$(\rp^3,L)$ and $(\rp^3,W_g(\alpha))$
are diffeomorphic.
\end{thm}
We say that an $MW_\lambda$-link $L$
corresponds to the partition $\alpha$
if $(\rp^3,L)$ and $(\rp^3,W_g(\alpha))$
are diffeomorphic.

The following result is the main theorem of this paper.
It strengthens the second half of Theorem \ref{thm-mwl}.
Two real algebraic links are said {\em rigidly isotopic}
if they can be connected with a one-parametric family  
of real algebraic links of the same degree.
\begin{thm}\label{thmain}
If $L$ and $L'$ are $MW_\lambda$-links
of the same chirality corresponding to the same
partition $\alpha$ 
then $L$ and $L'$ are rigidly isotopic.
\end{thm}

This theorem is a particular case of 
Theorem \ref{thm-gen}
(see Section \ref{s-l}) where we describe rigid isotopy classes of nodal curves
of any genus. The proof of Theorem 4 is based on the theory of divisors
on nodal Riemann surfaces and on the properties of $MW_\lambda$-curves
established in \cite{MO-mw}.

Note that in the case of rational curves 
Theorem 2 corresponds to \cite[Theorem 1]{MO-short}
with a simpler proof specific for genus 0.
In contrast to that, the only known to us way
to prove Theorem 3 is to deduce it from Theorem 4,
even in the case of rational curves.
In particular, this proof requires
\cite[Theorem 2]{MO-mw} in whole generality. 
Note that Theorem 4 provides another proof of
\cite[Theorem 3]{MO-mw}, i.e. 
existence of $MW_\lambda$-links isotopic to $W_g(\alpha)$
for any partition $\alpha=(a_0,\dots,a_g)$.

\section{Algebraic Hopf links}
\begin{defn}\label{iaHl}
A real algebraic link $L$ with $l$ 
real branches is called {\em an irreducible algebraic Hopf link}
if it is irreducible, its degree
is $l+2$, and each real branch
of $A$ is non-contractible in $\rp^3$.
\end{defn}
An irreducible algebraic Hopf link 
has minimal possible degree among all irreducible
real algebraic links with $l$ real branches such that
each real branch
is non-contractible in $\rp^3$ by Proposition \ref{Mldeg}.
By Proposition \ref{HMW} it is always a Hopf link topologically. 
\ignore{
\begin{rmk}
By Harnack's inequality \cite{Harnack} $g+1$ is the maximal possible
number of real branches of
a smooth irreducible curve $A$ of genus $g$.
Real curves with this maximality property
are known as {\em M-curves}. 
Extending this terminology, we call the real algebraic links
corresponding to M-curves {\em M-links}.
Note that all $MW_\lambda$-links are M-links. 
\end{rmk}
}
\begin{prop}\label{Mldeg}
Let $L\subset\rp^3$ be an irreducible real algebraic link 
of degree $d>1$ such that each component
of $L$ is non-contractible in $\rp^3$.
Then $d\ge l+2$.
\end{prop}
\begin{proof}
Take a pair of conjugate points on $A\setminus\R A\subset\pp^3$
and consider a real plane through this pair.
Since any plane in $\rp^3$ must intersect each non-contractible component
of $L$ we get at least $l+2$ intersection points.
\end{proof}
\ignore{
\begin{rmk}
Note that if we drop the condition that a real algebraic link $L$
with non-contractible components is an M-link then the conclusion
of Proposition \ref{Mldeg} is no longer true. For example,
it is easy to see that there exist real algebraic links $L$ obtained
as an intersection of a hyperboloid quadric and a real surface
of degree $d$ in $\rp^3$ such that $L$ consists of $d$ non-contractible
components. 
In the same time the genus of $g$ is $(d-1)^2$ by the adjunction formula.
\end{rmk}
}
\begin{prop}\label{HMW}
An irreducible algebraic Hopf link $L$ or its mirror image 
is an $MW_\lambda$-link
$W_{l-1}(1,\dots,1)$, and thus is topologically isotopic to a Hopf
link, i.e. the union of $l$ fibers of the (positive) 
Hopf fibration $\rp^3\to S^2$.
\end{prop}
\begin{proof}
The link $L$ satisfies to condition $(iii^t)$ of
Theorem 2 of \cite{MO-mw}. Namely, any
plane section in $\rp^3$ intersects $L$ in at least $d-2=l$
points (one for each component). 
\ignore{
We claim that $L$ satisfies to the condition $(ii)$ of
Theorem 2 of \cite{MO-mw}, i.e. any real plane tangent to $L$
has only real intersections with $A$.
To see this we note that any real plane $P$ intersects any
component of $L$ in odd number of points (counted with multiplicity).
If $P$ is tangent to a component of $L$ then $P$ must intersect 
this component yet in another point, or have a tangency of order 3.
Thus we already have at least $l+2$ real intersection
points of $P$ and $L$, and thus there may be no non-real
intersection points.
}
By Theorem 2 of \cite{MO-mw}
$L$ or its mirror image is topologically isotopic to $W_{l-1}(1,\dots,1)$
as $(1,\dots,1)$ is the only partition of $l$ into the sum of
$l$ positive integers. 
\end{proof}
\begin{lem}[cf. Theorem 2.5 of \cite{Huisman}]
\label{lem-ns}
An effective real divisor $D=\sum\limits_{j=1}^n a_jp_j$, $a_j>0$,
$p_j\in S$, on a real curve $S$ is non-special if at least $g$ distinct
real branches of $A$ intersect
$\{p_1,\dots,p_n\}$.
\end{lem}
\begin{proof}
By the Riemann-Roch Theorem, a divisor $D$ is special
if and only if the difference $K-D$ of the canonical
class $K$ and the divisor $D$ is representable 
with an effective divisor $D'$. 
Suppose that there exists a holomorphic form
whose zero divisor is $D+D'$.
But the zero divisor of a form
on each real branch of a curve must be even (if counted with multiplicities). Thus the degree of $D+D'$
is at least $2g$ while the degree of the canonical 
class is $2g-2$.
\end{proof}

By a {\em real Riemann surface} we mean an (irreducible)
Riemann surface $S$ enhanced with an antiholomorphic
involution $\conj$. 
Its {\em real locus} $\R S$ is the fixed locus
of $\conj$.
By Harnack's inequality, 
$g+1$
is the maximal possible
number of real branches of $S$ if
$g$ is the genus of $S$.
Real curves with this maximality property
are known as {\em $M$-curves}.

Recall that all effective divisors linearly
equivalent to an effective
divisor $D$ of degree $d$
form a projective space
$|D|\approx\pp^r$ whose dimension is called
the rank $r(D)$ of the divisor.
If the divisor is base point free, i.e.
every point of $S$ is {\em not} contained
in some divisor from $|D|$ then
in addition we have 
the map
\[
\iota_D:S\to |D|^\vee\approx\pp^r
\] to the space
$|D|^\vee$ projectively dual to $|D|$.
The point $x\in S$ is mapped
to the hyperplane in $|D|$ consisting
of divisors containing $x$.
If $S$ is a real Riemann surface and $D$ is 
a divisor invariant with respect to the conjugation
then we say that $D$ is a {\em real divisor}.
In this case the map $\iota_D$ is also real,
i.e. equivariant with respect to the conjugation
and the image $A=\iota_D(S)\subset\pp^r$ 
is a real projective curve.

\begin{prop} \label{prop-iHl}
Let $S$ be a real Riemann surface of genus $g=l-1$
with $l$ real branches
(i.e. $S$ is an M-curve).
and $D$ be a real divisor of degree $l+2$ such that every 
real branch
of $S$ contains odd number of points from $D$ (counted 
with multiplicity). Then the algebraic curve 
$A=\iota_D(S)\subset \pp^3$ 
corresponding to $(S,D)$ is an irreducible algebraic Hopf link.
In particular, $A$ is non-singular, and $r(D)=3$.
\end{prop}
\begin{proof}
The proposition is the special case of
Corollary 2.8 of \cite{Huisman} for $r=3$.
\ignore{
Since $r(K-D)=-1$ by Lemma \ref{lem-ns}, the Riemann-Roch formula
\eqref{RR} implies that $r(D)=3$.
Suppose that $A\subset\pp^3$ is singular. 

If the singularity comes from a single real branch of $A$
then the local intersection number
of $A$ and the osculating plane at this singularity
(i.e. the tangent plane with the highest index of tangency)
is at least 4. As this plane also must intersect all other
real branches
of $A$ by the parity reason,
we get a contradiction with the Bezout theorem.

If the singularity is real, but comes as an intersection of
two distinct real branches then we may consider
a plane tangent to both branches at this singularity.
This plane has the intersection number of at least three 
with the adjacent components of $\R S$ and at least one
with all the other components.
We get a contradiction with the Bezout theorem once again.
Finally, if $A$ has a complex conjugate pair of singularities
we get a contradiction with the Bezout theorem by consideration
of a real plane passing through this pair.
}
\end{proof}

\begin{coro}
Irreducible algebraic Hopf links
of the same chirality are rigidly isotopic.
\end{coro}
\begin{proof}
Let $L,L'\subset\rp^3$ be two irreducible algebraic Hopf links
represented as $\iota_D(S)$ and $\iota_{D'}(S')$
for real
Riemann surfaces $S$ and $S'$ and real divisor
$D$ and $D'$ on them.
By Proposition \ref{HMW} the Riemann surfaces $S$ and $S'$
are $M$-curves. Therefore their quotients by the complex conjugation
are diffeomorphic to a disk with $g$ holes.
Thus $S$ and $S'$ can be deformed into each other through
a family of real curves.

Without loss of generality (by taking an appropriate real plane section
of the corresponding irreducible algebraic Hopf links)
we may assume that the real divisors $D$ and $D'$ contain
a single point on each real branch of $S$ and $S'$ as well as
a pair of complex conjugate points.
Thus we may enhance the deformation of $S$ to $S'$ with 
a deformation of $D$ into $D'$ in the space of real divisors
on real Riemann surfaces.
By Proposition \ref{prop-iHl} this deformation corresponds
to a deformation in the class of irreducible algebraic Hopf links,
i.e. a rigid isotopy.

Therefore, $L$ and $L'$ are rigidly isotopic up to a projective
equivalence. 
Since the group $PGL_4(\R)$ consists of two connected components,
$L$ and $L'$ are rigidly isotopic if and only if their invariant
$w_\lambda$ (which cannot vanish for a $MW_\lambda$-link) 
is of the same sign.
\end{proof}

\begin{rmk}
Note that 
Conjecture 3.4 of \cite{Huisman} 
is false.
Not only algebraic Hopf links, but all 
$MW_\lambda$-links are unramified in the sense
of \cite{Huisman} by the condition $(iii)^t$
of \cite{MO-mw}. Taking a real algebraic
link $L\subset\rp^3$ 
isotopic $W_g(a_0,\dots,a_g)$
with even $a_j$
we get an unramified curve whose real branches are
contractible in $\pp^3$. 
Existence of such link is ensured by Theorem 3
of \cite{MO-mw}.
\end{rmk}

\section{Nodal links}
Recall that a nodal projective curve $A\subset\pp^n$,
$n\ge 2$,
is a (complex) algebraic curve such that all of its
singularities are simple nodes, i.e. points of 
crossings of two non-singular local
branches with distinct
tangent lines. As before, $A$ is real if it is 
invariant with respect to the involution of complex
conjugation $\pp^n\to\pp^n$.
The real locus $\R A=\rp^n\cap A$
of a nodal curve is a disjoint
union of immersed circles and a finite
set $\R E\subset\R A$. The points of
$\R E$ are called
{\em solitary nodes of $\R A$}.
\begin{defn}
An infinite subset $L\subset\rp^3$
is called
an irreducible {\em nodal real algebraic link}
if there exists an irreducible
nodal real algebraic curve $A\subset\rp^3$
such that $L=\R A$, and 
$L$ is not contained in any
proper projective subspace of $\rp^3$.
In particular, $L\neq\emptyset$.
Two nodal real algebraic links are {\em rigidly isotopic}
if they can be connected with a 1-parametric isotopy
in the class of real algebraic links of the same degree.

A {\em node} of $L$ is a node of the curve $A$.
\end{defn}
Nodal real algebraic links provide a generalization
of real algebraic links.
As in the case of real algebraic links,
an irreducible nodal real algebraic link
$L\subset\rp^3$
uniquely determines $A\subset\pp^3$.
\begin{defn}
We say that a real algebraic link $L_1$
{\em degenerates
to a real nodal algebraic link} $L_0$
if there exists a continuous family
of real algebraic links 
$L_t\subset\rp^3$, $t\in[0,1]$, of constant degree
such that $L_t$ for $t>0$ is a (non-nodal)
real algebraic link.
In this case we also say that the nodal link
$L_0$ {\em can be perturbed} to the smooth link $L_1$.
The perturbation (resp. degeneration) is
called {\em equigeneric} if the genus 
stays constant.

A smooth perturbation of a nodal curve is  equigeneric
if each node is perturbed 
as in Figure \ref{figPerturb} (left) but
not as in Figure \ref{figPerturb} (right).
All degenerations and perturbations considered
in this paper are assumed to be equigeneric.
\end{defn}

\begin{figure}[h]
\includegraphics[height=12mm]{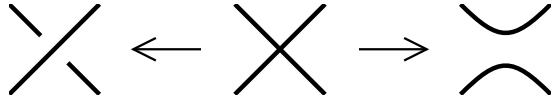}
\caption{
Equigeneric (on the left) and non-equigeneric
(on the right) perturbations of a spatial
nodal curve.
\label{figPerturb}}
\end{figure}

\ignore{
\begin{lem}
For a real algebraic link $A\subset\pp^3$,
and generic point $p\in\rp^3\setminus A$ 
the curve $\pi_p(A)\subset\pp^2$ is nodal.
\end{lem}
\begin{proof}
The union of all lines tangent to $A$ is a surface,
and thus nowhere dense in $\rp^3$.
Thus the 
\end{proof}
}
\begin{prop}\label{prop-mdeg}
Any irreducible real algebraic link 
$L=\R A\subset\rp^3$ of genus $g$ and degree $d$
degenerates to a nodal link with at least
$d-g-3$ nodes.
\end{prop}
\begin{proof}
By induction it is enough to assume that if $L=\R A$ is a nodal link
of genus $g$ and degree $d$ with $n<d-g-3$ nodes, then it
degenerates to a nodal link with at least one node more.

Let $p\in\rp^3$ be a generic point so that the image
$B=\pi_p(A)\subset\pp^2$
under the projection $\pi_p:\pp^3\setminus\{p\}$
is a real nodal planar curve.
By choosing $p$ near a line passing through two distinct
real points of $A$, we may assume that $B$ has a real node with
two real branches which is not a projection of a node of $A$.
This means that there are distinct points $x,y\in\R A$ such that
$$
    u=\pi_p(x)=\pi_p(y).
$$
We choose the coordinates so that $p=(0:0:0:1)$.

Let $\iota:S\to A$ be the normalization and
$D=\iota^{-1}(A\cap H)$ be the divisor on $S$
obtained as the intersection of $A$
and a generic real plane $H\subset\pp^3$.
The divisor $D$ is a real divisor of rank $r\ge d-g$.
The immersions of $S$ given by $A\subset\pp^3$ and
$B\subset\pp^2$ correspond to
a real projective 3-dimensional subspace of $|D|$
and a real projective plane inside it.

Consider the vector space $\Gamma(S,D)$ formed
by the sections
of the line bundle defined by $D$.
In the coordinates the curve $A\subset\pp^3$
is given by four linearly independent sections
$s_0,s_1,s_2,s_3\in\Gamma(S,D)$ such that
$s_0,s_1,s_2$ define the curve $B\subset\pp^2$.
The condition that $s_0,s_1,s_2,s$ define an immersion
with $n$ nodes
imposes $n$ linear conditions on $s\in\Gamma(S,D)$. 
Let $V\subset\Gamma(S,D)$ be the
space of all sections satisfying these conditions, in particular, we have
$s_3\in V$. Since $r(D)-n>3$, it follows that
there exists $s_4\in V$ not
contained in the linear span of $s_0,s_1,s_2,s_3$.
Taking $s_4$ close to $s_3$ we may assume that
$s_0,s_1,s_2,s_4$ define a curve $A'\subset\pp^3$
without a node at $\pi_p^{-1}(u)$.

With the help of the plane sections of
the curves $A,A'$
passing through one point of intersection
with $\pi_p^{-1}(u)$ but not the other
we can choose $s'_j\in\Gamma(D,S)$, $j=3,4$,
with the following properties:
\begin{enumerate}
\item
$s_j$ is the linear combination of $s_0,s_1,s_2,s'_j$
(in particular, the curve defined by $s_0,s_1,s_2,s'_3$
is projectively equivalent to $A$);
\item 
$s'_3(\tilde x)=0=s'_4(\tilde y)$;
\item
$s'_3(\tilde y)\neq 0\neq s'_4(\tilde x)$.
\end{enumerate}
Here $\tilde x=\iota^{-1}(x),\tilde y=\iota^{-1}(y)\in S$.
Note that (1) implies
linear independence of $s_0,s_1,s_2,s'_3,s'_4$.

\ignore{
Let $p\in\rp^3$ be a generic point so that
the image
$B=\pi_p(A)\subset\pp^2$ 
under the projection
$\pi_p:\pp^3\setminus\{p\}\to\pp^2$
is a real nodal planar curve.
By the adjunction formula
$B$ has $N_d-g\ge d-g-3$ nodes.
We choose the coordinates so that 
$p=(0:0:0:1)$.

Let $D=A\cap H$ be the divisor on $S=A$
obtained as the intersection of $A$
and a generic real plane $H\subset\pp^3$.
The divisor $D$ is a real divisor
of rank
$r\ge d-g$.
The embedding $A\subset\pp^3$ and the immersion
$B\subset\pp^2$ correspond to 
a real projective 3-dimensional subspace of $|D|$
and a real projective plane inside it.

Consider the vector space $\Gamma(S,D)$
formed by sections of the line bundle defined by $D$.
In coordinates the embedding $A\subset\pp^3$
is given by four linearly independent sections
$s_0,s_1,s_2,s_3\in\Gamma(S,D)$
such that
$s_0,s_1,s_2$ define the immersion $B\subset\pp^2$.
If $r(D)>3$ then 
there exists $s_4\in\Gamma(S,D)$ not
contained in the linear span of $s_0,s_1,s_2,s_3$.
Adding a multiple of $s_3$ to $s_4$
if needed we may assume that $s_4(x)=0$
while $s_4(y)\neq0$
for $x\neq y\in A$ such that
\[
u=\pi_p(x)=\pi_p(y)\in\R B
\]
is a real non-solitary node of $\pi_p(A)$
if the number of real non-solitary
nodes of $\pi_p(A)$ is larger
than the number of  nodes of $A$
(in the beginning, the curve $A$ is smooth,
and does not have nodes). 
}
The image of $S$ in $\pp^3$ under
$[s_0:s_1:s_2:s'_3+ts'_4]$, $t\in\R$,
is a real curve $A_t\subset\pp^3$ with $A_0=A$.
Since the first three sections define
a nodal immersion to $\pp^2$, all singularities
of $A_t\subset\pp^3$ (if any) are nodes.
By linear independence of $s_0,s_1,s_2,s_3,s_4$,
the curve $A_t$ cannot be contained
in a plane in $\pp^3$.

By (2) and (3) there exists $t_u\in\R$ such that
$A_{t_u}$ has a node over $u$ while $A_0$ does
not have a node over $u$. Thus $A$ can be
degenerated to a curve with at least one
node (though not necessarily over $u$ since
another node may appear during the deformation
of $t$ from 0 to $t_u$, in particular the
new node could be solitary, or even a pair
of complex conjugate nodes).
\end{proof}

\section{Nodal $MW_\lambda$-links and chord diagrams}
\begin{defn}\label{def-nMW}
An irreducible nodal algebraic link is called
a nodal $MW_\lambda$-link if it can be (equigenerically) perturbed
to a (smooth) $MW_\lambda$-link.
\end{defn}
\begin{lem}\label{lem-nodes}
All nodes of a nodal $MW_\lambda$-link are real and
non-solitary.
\end{lem}
\begin{proof}
Let $L_0$ be a nodal $MW_\lambda$-link and $L_t$,
$t\in (0,1]$
be smooth $MW_\lambda$-links degenerating to $L_0$.
If $L_0$ has a non-real node $\nu$ then the planar
curve $\pi_p(L_\epsilon)\subset\pp^2$
must have a non-real node near $\nu$ for small $\epsilon>0$.
Thus $L_\epsilon$ can not be
maximally writhed.

Suppose that $L_0=\R A_0$
has a solitary real node $\nu$.
Let $p\in\rp^3\setminus\R A_0$ be a point on
a line $\ell$ connecting a pair of complex conjugate
points of $A_0$.
The projection $\pi_p(A_{\epsilon})\subset\pp^2$
has two solitary nodal points: one corresponding
to $\nu$, and one corresponding to $\ell$.
We get a contradiction to Condition $(iv^a)$ of
\cite[Theorem 2]{MO-mw}.
\end{proof}

Many of the properties of $MW_\lambda$-links 
found in \cite{MO-mw} hold also for any
nodal $MW_\lambda$-link $L_0=\R A_0\subset\rp^3$.
\ignore{
\begin{lem}
For any $p\in \rp^3\setminus L_0$ there exists
a line $\ell\in\rp^3$, $\ell\ni p$, such that
$L_0$ is hyperbolic with respect to $\ell$.
\end{lem}
\begin{proof}
Note that it suffices to prove the lemma for smooth
$MW_\lambda$-links $L_t\subset\rp^3$, $t\in (0,1]$,
as reality of intersection 
points is conserved under taking the limit. 
By \cite[Lemma 3.10]{MO-mw} the real tangent surface
\[
\R T(L_t)=\bigcup\limits_{q\in L_t} \R T_q,
\]
where $\R T_q$ is the tangent line to $L_t$ at $q$,
is a union of topologically embedded hyperboloids.
If $p\in\R T_q$ then the lemma follows
from \cite[Theorem 2(ii)]{MO-mw} (for $\ell=\R T_q$).
If $p\in\rp^3\setminus \R T(L_t)$ then the lemma
follows from \cite[Lemma 3.11]{MO-mw}
(cf. \cite[Figure 3]{MO-mw}).
\end{proof}
}

Let $q\in A_0\setminus\R A_0$. Denote with
$\ell_{q\bar q}\subset\rp^3$ the line whose complexification
in $\pp^3$ passes through $q$ (and thus also through $\bar q$). 
\begin{lem}\label{lemqq}
For any $q\in A_0\setminus\R A_0$ we have
$\ell_{q\bar q}\cap L_0=\emptyset$.
\end{lem}
\begin{proof}
By Lemma \ref{lem-nodes} $L_0$ has no solitary nodes.
Thus at a point of $\ell_{q\bar q}\cap L_0$
there exists a tangent plane $H\subset\rp^3$
containing $\ell$. Existence of such plane contradicts
to the property \cite[Theorem 2 (ii)]{MO-mw} of
$MW_\lambda$-links which is clearly
conserved under passing to the limiting
nodal links.
\end{proof}

\begin{cor}
For each real branch $K$
the linking number 
\begin{equation}
a(K)=2|\lk(\ell_{q\bar q},K)|
\end{equation}
(taken twice)
in $\rp^3$ is independent of $q\in A_0\setminus\R A_0$.
\end{cor}
We have 
\begin{equation}\label{a-comp}
\sum\limits_K a(K)=d-2
\end{equation}
by \cite[equation (3)]{MO-mw}
as the integer numbers $a(K)$ agree with the numbers
introduced in \cite[Lemma 4.12]{MO-mw} under the degeneration
of $L_t$ to $L_0$.
\ignore{
For $p\in \tilde L_0$ we may consider the tangent line
$T_p\subset\rp^3$.
The proof of the following lemma is identical to that of
\cite[Lemma 3.10]{MO-mw}.
\behin{lem}
If $p,q\in \tilde L_0$ with $\nu(p)\neq\nu(q)$ 
then $T_p\cap T_q=\emptyset$.
\end{lem}
We denote with
\[
T=\bigcup\limits_{p\in \tilde L_0}T_p
\]
the tangent surface to $L_0$.

\begin{prop}
For a generic plane $P\subset\rp^3$ 
we have \[\#(K\cap\nu^{-1}(P\cap L_0))\ge s(K).\]
\end{prop}
\begin{proof}
Since $P$ is not tangent to $L_0$ the intersection
$P\cap T$ is ..
\end{proof}
}

When speaking of nodal curves, we shall
use the language of chord diagrams.
\begin{defn}\label{def-cd}
A {\em chord diagram} $(X,\Lambda)$
is a pair consisting of
the disjoint union $\Lambda$
of $l$ oriented circles, and the topological
space $X$ obtained as the result of attachment
of $\delta\ge 0$ intervals (chords) along
their endpoints to $\Lambda$ so that
distinct endpoints are glued to distinct points
of $\Lambda$. We consider $(X,\Lambda)$
up to homeomorphisms of the pair
preserving the orientation of $\Lambda$.
We denote the complement of the chord endpoints
in $\Lambda$ with $\Lambda^\circ\subset\Lambda$.
\end{defn}
	
\begin{defn}\label{planar-la}
A chord diagram $(X,\Lambda)$ is called {\em planar} 
if it can be embedded to the disjoint union $\Delta$ of
$l$ copies of 2-disks so that $\Lambda$ is mapped to
the boundary of the disjoint union.
\end{defn}
Clearly no chord of a planar chord diagram can
connect points from different components of $\Lambda$.
Thus a planar diagram for arbitrary $l$
 is  a disjoint union of $l$ connected
planar chord diagrams.

By a loop $\lambda$ in a chord diagram $(X,\Lambda)$
we mean a simple closed loop in $X$. If  $(X,\Lambda)$
is planar then
we refer to these loops as {\em planar loops}.

\begin{defn}
The chord diagram $X(L_0)$ of a
a nodal link $L_0$ is the chord diagram
obtained by attaching
a chord to $\R\tilde A_0$ connecting
the respective points of the normalization
for every node of $\R A_0$.
\end{defn}
For a loop $\lambda\subset L_0$ the number 
\begin{equation}\label{alambda}
a(\lambda)=2|\lk(\lambda,\ell_{q\bar q})|
\end{equation}
is independent of the choice of $q\in A_0\setminus\R A_0$
by Lemma \ref{lemqq}.

\begin{prop}\label{gd}
The chord diagram of a nodal $MW_\lambda$-link $L_0$
is planar.
\ignore{
Each node of a nodal $MW_\lambda$-link $L_0$ corresponds
to a self-intersection point of the same component
of $L_0$. 

Furthermore if $\mu,\nu$ are two nodes
corresponding to self-intersections of the same
component $K\subset L_0$ then $\nu$ corresponds to
a self-intersection point of the same component
of the normalization of $K\setminus\{\mu\}$.
}
\end{prop}
\ignore{
The second statement of the lemma may be reformulated
with the help of the so-called {\em chord}
(or {\em Gau{\ss}})
diagrams. Namely, let $\tilde K\to K$
be the normalization of the component $K$
(i.e. the restriction of the normalization
$\rho:\tilde A\to A$ to the corresponding component
of $\tilde K\subset\R\tilde A$).
The circle $\tilde K$ bounds a disk $D$.
For each node $x\in K$ we connect
the two points of $\rho^{-1}(\nu)$ with a chord
$\chi_\nu\subset D$ (a straight interval).
The union 
\[
\Delta=\tilde K\cup\bigcup\limits_\nu \chi_\nu
\]
over all nodes $x$ is called {\em the chord
diagram} of $K$.
The lemma claims that the chords in the case
of a nodal
$MW_\lambda$-link do not intersect.
}
\begin{proof}
\ignore{
Let $\ell\subset\rp^3$
be a an oriented line whose complexification
passes through a pair of complex conjugate
points of $A$. 
Let $K$ be a component of the normalization
$\tilde L_0$
of $L_0$.
}
Let $L_0=\R A_0$ be a nodal $MW_\lambda$-link and
$\nu:\tilde A_0\to A_0$ be the normalization of $A_0$.
Let $x$ be a node of $L_0$ and $H$ be
the real plane tangent to both local 
branches
of $A_0$ at $x$.
Consider the intersection number $n(K)$ of $H$ and $A_0$ 
along a real branch $K\ni x$.
Namely, $n(K)$ is the intersection number in $\pp^3$ of the
complexification of $H$ and $\nu(U)$, where $U$ is
a small neighborhood of $\tilde K$ in $\tilde A_0$
(here $\tilde K$ is the connected component of $\R\tilde A_0$
such that $\nu(\tilde K)=K$).
Consider
a point $q\in \nu(U)\setminus K$ close to $x$.
Then the line $\ell_{q\bar q}$ is close to the tangent line
at $x$ to $K$.
The plane $H$ may be perturbed to a real plane $H'\supset\ell_{q\bar q}$.
	We have $\#(H'\cap K)\ge 
2|\lk(\ell_{q\bar q},K)|=a(K)$ since $H'\subset\ell_{q\bar q}$.
But in addition the complexification of $H'$ intersects $A_0$
	at $q$ and $\bar q$. Thus $n(K)\ge\#(H'\cap K)+2\ge a(K)+2$.

If two distinct real branches $K_1,K_2$ meet at $x$
then taking the sum over all real branches we get
the inequality
\[
\sum n(K)\ge \sum a(K)+4=d+2
\]
contradicting to the Bezout theorem for the intersection of $H$ and
$A_0$.

A node $y\in L_0$ where both local branches correspond
to the same real branch $K$ defines
a subdivision of $K$ into two loops
$\lambda_j\subset K$, $j=1,2$.
Each loop $\lambda_j$ is the closure
of the image of an arc of $\tilde K\setminus\nu^{-1}(y)$
under $\nu$. 
If the chord diagram $X(L_0)$ is not planar
then 
we may assume that $K,x,y$ where chosen so
	that $\lambda_1$ and $\lambda_2$ pass through $x$.
	We have $a(\lambda_1)+a(\lambda_2)=a(K)$; see \eqref{alambda}.

Denote
with $n(\lambda_j)$ the intersection number of the complexification
	of $H$ and $\nu(U_j)$ for a small neighborhood 
	$U_j$ of $\nu^{-1}(\lambda_j)$ in
$\tilde A_0$. We get $n(\lambda_j)\ge a_j+2$
as above.
If $y\notin H$ then 
	$\nu(U_1\cap U_2)$ is disjoint from $H$,
	thus
	$n(K)=n(\lambda_1)+n(\lambda_2)
\ge a(K)+4$,
and we get a contradiction to the Bezout Theorem as before.

If $y\in H$ then $n(K)=n(\lambda_1)+n(\lambda_1)-n(y)$,
where $n(y)$ is the local intersection number of $A_0$ and
the complexification of $H$
at $y$. But in this case we have 
$n(\lambda_1)+n(\lambda_1)\ge a(K)+4+n(y)$ so we get a contradiction
to the Bezout Theorem anyway.
\end{proof}

\begin{defn}\label{def-degchord}
The {\em degree-chord} diagram of a nodal $MW_\lambda$-link $L_0$
is a chord diagram enhanced with a number $a(\lambda)$ defined
for any planar loop of $L_0$.
The number $a(\lambda)$ is 
the {\em degree} of the planar loop $\lambda$
of the planar chord diagram $X(L_0)$.
\end{defn}

\section{Divisors on singular Riemann surfaces}\label{nRS}
Let $S$ be a (closed connected irreducible)
Riemann surface. 
\begin{defn}
The {\em chord diagram} $X_S$ on $S$ is
the topological space obtained by attaching to $S$
a finite number of 1-cells ({\em chords})
$I\approx [0,1]$ in an acyclic way,
i.e. so that the union of all chords
is a forest (in particular no chord is
a loop in $X_S$).
Let $\delta$ be the number of chords.
The {\em singular surface} $\Sigma=\Sigma(X_S)$
is obtained by contracting each chord to a point. 
\end{defn}

If the boundaries of the chords are disjoint in $S$,
we say that $X_S$ is {\em nodal}.
Then we may consider $\Sigma$ as an abstract nodal curve.
Otherwise $\Sigma$ has $k$-fold points with $k>2$.

The theory of divisors and their linear systems
on $\Sigma$ is very similar to its counterpart
on ordinary Riemann surfaces.
Denote with $N\subset\Sigma$ the singular locus,
i.e. the image of the union of chords
under the contraction map $c:X_S\to\Sigma$.
Let $\tilde N=c^{-1}(N)\cap S$. As usual,
for a meromorphic function $f:S\to\cp^1=
{\mathbb C}\cup\{\infty\}$
we denote with $(f)$ its divisor on $S$ composed
of its zeroes and poles with multiplicities.
Let $\Sigma^*=\Sigma\setminus N=S\setminus\tilde N$.

Let us fix an effective divisor
$D=\sum\limits_{j=1}^n a_jp_j$, $a_j\in\Z_{>0}$,
with $p_j\in S\setminus\tilde N$.
Denote with $V_D$ the set of meromorphic
functions $f:S\to\cp^1$ such that $(f)\ge -D$
(i.e. $f$ has, at worst, a pole of order $a_j$
at $p_j$, and holomorphic elsewhere) and
\[
c(x)=c(y) \implies f(x)=f(y),\ \forall x,y\in S.
\]
The space $V_D$ is a vector space.
We say that a divisor $D'$ on $S$
is $\Sigma$-equivalent to $D$,
and write $D\sim D'$, if $D-D'=(f)$ for $f\in V_D$.
Recall that we require $D$ to be disjoint from $\tilde N$
while we do not impose this requirement on $D'$.
In particular, the relation $\sim$ is not an equivalence
relation (cf. Remark \ref{rem-odivisorax}). 
Let $|D|_\Sigma$ be the set of 
effective divisors $\Sigma$-equivalent to $D$.
It can be naturally identified with the
projectivization of $V_D$.

The Riemann-Roch theorem implies that
for each chord $I\subset X_S$
there exists a meromorphic form $\omega_I$ on $S$,
holomorphic on $S\setminus\dd I$ and
having a pole of the first order at
each point of $\dd I$. 
Clearly, the residues of $\omega_I$
at the two points of $\dd I$ must be opposite. 
Multiplying this form by a constant,
we may assume that
the residues of $\omega_I$ at $\dd I$
are $\pm 1$. 
Furthermore, for a choice of a system of $a$-cycles,
i.e. of $g$ simple loops $\alpha_j$, $j=1,\dots,g$,
in $S$ such that 
$S\setminus\bigcup_{j=1}^g\alpha_g$ is 
a planar domain, we may assume that
\begin{equation}\label{omI}
\int\limits_{\alpha_j}\omega_I=0
\end{equation}
by adding a holomorphic form on $S$.
 
Denote with $\Omega(\Sigma)$ the vector space 
generated by the holomorphic forms on $S$ and
the forms $\omega_I$ for the chords $I\subset X_S$.
This is the space of meromorphic forms on $S$
holomorphic on $S\setminus\tilde N$, 
having at worst poles of the first order at
the points of $\tilde N$, and such
that the sum of residues at all points of $\tilde N$
corresponding to the same point of $N$ vanishes.
Thus $\Omega(\Sigma)$ depends only on $\Sigma$,
and not on a particular representation of $\Sigma$
by $X_S$.
We have 
\[
\dim \Omega(\Sigma)=g+\delta.
\]
We refer to the elements of $\Omega(\Sigma)$
as {\em $\Sigma$-holomorphic form}.
%

Any element of $H_1(\Sigma^*)$ defines a functional
on $\Omega(\Sigma)$ through integration.
The {\em Jacobian} of $\Sigma$ is defined as
the quotient
\[
\operatorname{Jac}(\Sigma)=\Omega^*(\Sigma)/H_1(\Sigma^*)
\] 
of the dual vector space to $\Omega(\Sigma)$
by the image of $H_1(\Sigma^*)$.
It is a $(g+\delta)$-dimensional complex variety
homeomorphic to $(S^1)^{2g}\times
({\mathbb C}^\times)^\delta$.
The divisor $D$ defines the {\em Abel-Jacobi map}
\[
\alpha_D:\operatorname{Sym}^n(\Sigma^*)\to
\operatorname{Jac}(\Sigma)
\]
by associating to $D'$
the functional $\omega\mapsto\int\limits_\Gamma\omega$,
where $\Gamma$ is a 1-chain in $\Sigma^*$ with
$\dd\Gamma=D'-D$.

\begin{thma}[the Abel-Jacobi Theorem for $\Sigma$]
\label{AJthm}
For an effective divisor $D'$ supported on
$\Sigma^*$
we have $D'\in |D|_{\Sigma}$ if and only if $\alpha_D(D')=0
\in\operatorname{Jac}(\Sigma)$.
\end{thma}
\begin{proof}
Suppose that $D'\in |D|_{\Sigma}$, i.e.  
$(f)=D'-D$ with $f\in V_D$.
By the conventional Abel-Jacobi theorem  
we have $\int\limits_\Gamma \omega=0$
for any holomorphic 1-form $\omega$ in $S$.
Suppose that $I$ is a chord in $X_S$ and 
$x_+,x_-\in S$ are the endpoints of $I$.
Let $\omega_I$ be the meromorphic form
on $S$, holomorphic on $S\setminus\{x_+,x_-\}$,
with the simple poles of residue $\pm 1$ at $x_\pm$,
and satisfying to \eqref{omI}.
Multiplying $f$ by a scalar we may assume that
$f(x_+),f(x_-)\in {\mathbb C}$ are non-real numbers.
The following equality is known as a reciprocity law 
for the Abelian differentials of the third kind
(cf. e.g. \cite[Lecture 6, Lemma 3]{Dubrovin}) 
\begin{equation}\label{recipro}
\int\limits_\Gamma\omega_I=
\int\limits_{\Gamma_I}\frac{df}{f}. 
\end{equation}
Here $\Gamma=f^{-1}(\R_{\ge 0}\cup\{\infty\})$ is
oriented so that $\dd\Gamma=D'-D$ while
$\Gamma_I$ is a path in $S$ from $x_-$ to $x_+$
disjoint from $\Gamma$.
To prove \eqref{recipro} it suffices to integrate
the 1-form $\Log(f)\omega_I$ over a small contour $\gamma$
going around $\Gamma$ for a holomorphic branch 
of $\Log(f)$ in $S\setminus\Gamma$.
Since the values of $\Log(f)$ at the two sides
of $\Gamma$ differ by $2\pi i$,
we see that $\int\limits_\gamma\Log(f)\omega_I$ 
equals to $2\pi i$ times the left-hand side
of \eqref{recipro}.
By Cauchy's residue formula
$\int\limits_\gamma\Log(f)\omega_I$ 
equals to $2\pi i$ times the right-hand side of \eqref{recipro}.

Since $f\in V_D$, the right-hand side
of \eqref{recipro} is zero modulo $2\pi i\Z$,
and thus $\alpha_D(D')=0$.
Conversely, suppose that $\alpha_D(D')=0$.
By the conventional Abel-Jacobi theorem there exists a
meromorphic function $f:S\to\cp^1$ with $(f)=D'-D$.
By \eqref{recipro} the integral $\int\limits_{\Gamma_I}\frac{df}{f}$
vanishes (modulo $2\pi i\Z$) for any chord $I$.
Thus $f\in V_D$.
\end{proof}

Denote with $\Omega_{-D}(\Sigma)\subset\Omega(\Sigma)$
the vector space
formed by the $\Sigma$-holomorphic forms vanishing in all points of $D$. 
\begin{coro}[The Riemann-Roch Theorem for $\Sigma$]
\label{RRthm}
We have 
\[
\dim |D|_\Sigma = n-g-\delta+\dim\Omega_{-D}(\Sigma).
\]
\end{coro}
\begin{proof}
Consider the differential $d\alpha$ of the Abel-Jacobi map
at $D\in \operatorname{Sym}^n(\Sigma^*)$.
It is a linear map from the $n$-dimensional vector space
composed as the direct sum of the tangent spaces of $S$ at the
points of $D$ (with multiplicities) to 
the $(g+\delta)$-dimensional space $\Omega^*(\Sigma)$. 
The cokernel of this differential coincides with the space $V_D$.
According to Theorem \ref{AJthm} the kernel is $V_D$.
\end{proof}

We say that $D$ is {\it non-special}
in $\Sigma$ if $\dim |D|_\Sigma = n - g - \delta$,
i.e., if $\Omega_{-D} (\Sigma) = 0$ (recall that $\supp D\subset\Sigma^*$). 

Consider the chord diagram $X^I_S$
obtained from $X_S$ by removing a chord $I$
and the corresponding singular surface
$\Sigma^I$.
\begin{lem}\label{lem-DI}
If a divisor $D=\sum\limits_{j=1}^n a_jp_j$,
$p_j\in\Sigma^*$, is non-special in $\Sigma$
then it is also non-special in $\Sigma^I$.
\end{lem}
\begin{proof}
By definition we have $\Omega(\Sigma^I)\subset\Omega(\Sigma)$ whence
$\Omega_{-D}(\Sigma^I)\subset\Omega_{-D}(\Sigma)$.
Thus the hypothesis of $\Omega_{-D}(\Sigma)=0$
implies the conclusion $\Omega_{-D}(\Sigma^I)=0$. 
\end{proof}

With each effective divisor $D$ such that
$\supp D\subset \Sigma^*$ we associate
the map
\begin{equation}\label{n-iota}
  \iota_D : S \to |D|_\Sigma^\vee,
\end{equation}
sending $x$ to the hyperplane of $V_D$ corresponding to divisors
$D' \in |D|_\Sigma$ with $x \in D'$.
This mapping evidently factors through a mapping $\Sigma\to |D|_\Sigma^\vee$.

As in the classical theory, any holomorphic mapping $f:S\to\pp^k$
which factors through $\Sigma$ is a composition of $\iota_D$
with a linear projection $\pi:|D|_\Sigma^\vee\dashrightarrow\pp^k$
where $D$ is
the pull-back of a generic plane section of $f(S)$. The image of
$\pi$ is the minimal subspace of $\pp^k$ containing $f(S)$.

\begin{rem}\label{rem-odivisorax}
When a smooth curve is embedded to a complete linear system $|D|$,
the embedded curve can be completely recovered (up to projective
transformation) by any hyperplane section. 
Nevertheless, the situation is very different
in the case of nodal curves.
If a hyperplane section passes through some nodes,
then it does not determine the embedding. 
This is the reason why we have supposed that
$\supp D\subset\Sigma^*$ in this section.
\end{rem}

\section{Real singular Riemann surfaces}
A chord diagram $X_S$ on a Riemann surface $S$
is {\em real algebraic} if the Riemann
surface $S$ is real
(i.e., endowed with an antiholomorphic involution
 $\conj:S\to S$),
and the
boundaries of all chords are contained
in $\R S=\operatorname{Fix}(S)$.
Then we may consider
\[
\R X_S=X_S\setminus(S\setminus\R S)
\]
(the union of $\R S$ with all the chords).
If $X_S$ is nodal then the pair $(\R X_S,\R S)$
is the chord diagram in the sense of 
Definition \ref{def-cd}.
Otherwise, it can be considered as a degeneration of such diagram.

We say that a real algebraic chord diagram $X_S$ is planar if 
$\R X$ can be embedded into the disjoint union of
disks in such a way that $\R S$ is mapped homeomorphically
onto the boundary. Note that in the case when the chord
boundaries are disjoint, this definition agrees with
Definition \ref{planar-la}.
We have $l+\delta$
planar loops in $\R X_S$
(here $l=b_0(\R S)$) that correspond to the components
of the complement of the image of this embedding.

\begin{lem}\label{evenzeroes}
Suppose that $X_S$ is a real planar chord diagram
and $\omega\in \Omega(\Sigma)$ is a real form on $S$ whose zeroes are
disjoint from $\tilde N$.
Then every planar loop in $\R X_S$
contains an
even number of zeroes of $\omega$ (counted with multiplicities).
\end{lem}
\begin{proof}
In the complement of its zeroes,
a real form $\omega$ defines an orientation of the underlying curve.
Since the residues of $\omega$ at the endpoints of each chord are
opposite, this orientation agrees with an orientation of the planar
loop near the chord. Thus the lemma follows from the orientability of a circle.
\end{proof}
\begin{coro}\label{nsD}
An effective real divisor $D=\sum\limits_{j=1}^n a_jp_j$, $a_j\in\Z_{>0}$,
$p_j\in \Sigma^*$, is non-special, if $X_S$ is planar
and at least
$g+\delta$ distinct
planar loops of $X_S$ intersect $\supp(D)=\bigcup\limits_{j=1}^n\{p_j\}$.
\end{coro}
\begin{proof}
Suppose that $\omega\in\Omega_{-D}(\Sigma)$. If $\omega$ does not have
poles at the boundary of a chord in $X_S$ then we may remove this chord
and the statement follows by induction from the corresponding statement
for the resulting diagram in $(\delta-1)$ chords.

Otherwise $\omega$ has the total of $2g+2\delta-2$ zeroes
in $\Sigma^*$ (counted with multiplicities).
But Lemma \ref{evenzeroes} implies existence of at least $2g+2\delta$
zeroes contained in the given $g+\delta$ planar loops.
\end{proof}

Consider a real algebraic nodal curve $A_0\subset\pp^3$
endowed with a fixed orientation of its real branches
and its real {\em equigeneric perturbation} $A_t\subset\pp^3$,
(cf. Definition \ref{def-nMW} but now $A_t$
is not required to be smooth) for small $t>0$.
We say that $A_t$ is {\em positive}
(resp. {\em negative}) at a node $x\in A_0$
if the perturbed local branches form
a positive (resp. negative)
crossing according to Figure 1
under a generic projection.

Note that the signs of $A_t$ at all nodes 
are invariant if we reverse the orientation of $A_0$.
Thus we may choose
any orientation in the case if $A_0$ is rational.
In the case when $A_0$ is of type I (the case of this
paper) we take a complex orientation of $\R A_0$.

\begin{lem}\label{positive-perturb}
Let $\varphi:S\to\pp^3$ be an analytic mapping which factors through an
injective mapping $\Sigma\to\pp^3$
and $D$ be the pull-back of a real plane section.
Suppose that $A_0=\varphi(S)$ is a real
nodal curve, $\supp(D)\subset \Sigma^*$, $D$ is
non-special in $\Sigma$, and $\dim |D|_\Sigma \ge 3$.
Let $I_1,\dots,I_k$ be some chords of $X_S$.

Then for any orientation on $\R S$ there exists
a real equigeneric perturbation $A_t$ of $A_0$
which has any prescribed signs at the nodes
corresponding to $I_1,\dots,I_k$ and
which keeps all the other nodes.
\end{lem}
\begin{proof}
Inductively with the help of Lemma \ref{lem-DI},
it suffices to prove this lemma
for $k=1$, i.e. that we may perturb a single node
of $A_0$ equigenerically, keeping all the other
nodes, and choosing any prescribed sign for
the perturbation of this node.
Since $A_0\subset\pp^3$ is nodal, its projection
onto $\pp^2$ from a generic point of $\pp^3$
is a planar nodal curve $B\subset\pp^2$.
Let $s_0,s_1,s_2,s_3\in V_D(\Sigma)$ the sections
defining the curve $A_0$, and such that 
$s_0,s_1,s_2$ define the curve $B$
(cf. the proof of Proposition \ref{prop-mdeg}).

By Corollary \ref{RRthm},
$$\dim V_D(\Sigma') > \dim V_D(\Sigma),$$
where $\Sigma'$ is the singular surface
corresponding to the diagram $X^{I_1}$,
i.e. when the chord $I_1$ is removed.
Let $s\in V_D(\Sigma')\setminus V_D(\Sigma)$.
Then $s_0,s_1,s_2,s$ define an analytic mapping
$S\to\pp^3$ that factorize through $\Sigma'$ 
but not through $\Sigma$.
Thus $s_0,s_1,s_2,s_3+ts$, $t\in\R$,
define the required
perturbation of $A_0$ where
different signs of $t$ correspond to different
signs of the node perturbation.
\ignore{
Let $X'_S$ be the diagram obtained from $X_S$ by removing
the chords $I_1,\dots,I_k$ and $\Sigma'$ be the corresponding singular
surface. By iterating Lemma \ref{lem-DI}, we see that
$D$ is non-special in $\Sigma'$, hence by
the Riemann-Roch Theorem
(Corollary 5.3) we have
\begin{equation}\label{eqRR}
       \dim V_D(\Sigma') - \dim V_D(\Sigma) = k.         \end{equation}
Furthermore, we have
$$V_D(\Sigma)=\bigcap\limits_{j=1}^k
V_D(\Sigma'_{I_j}),$$
where $V_D(\Sigma'_{I_j})\subset V_D(\Sigma')$
is the hyperplane corresponding to inserting back
the node corresponding to $I_j$ in $\Sigma'$.
Because of \eqref{eqRR} these hyperplane must
intersect transversely.

Let us fix a generic point $p$ and choose real coordinates in $\pp^3$ so that
$p=(0:0:0:1)$. We may assume that the plane projection of $A$ from $p$
is a nodal curve. Consider the mapping $V_D(\Sigma')\to
{\mathbb C}^k$ defined by
$f\mapsto(z_1,\dots,z_k)$ where $z_j=f(y_j)-f(x_j)$.
By definition of $V_D$, the kernel of this mapping is $V_D(\Sigma)$ while..

By combining this fact with \eqref{eqRR}
we conclude that this mapping is surjective,
i.e., we may choose $f\in V_D(\Sigma')$ such that the differences
$f(y_j)-f(x_j)$ have any prescribed signs.
By correcting the last coordinate of $\varphi$ with
help of $f$ we obtain the desired perturbation of $A$.
}
\end{proof}

\section{Nodal Hopf links}
\begin{defn} 
A nodal $MW_\lambda$-link is called
a {\em nodal Hopf link} if it has $d-3-g$ nodes.
\end{defn}

By \eqref{a-comp} the degree of any planar loop in 
a nodal Hopf link is 1. Conversely, as
the number of planar loops is $d-2$,
we get the following characterization
of nodal Hopf links. 
\begin{prop}\label{nHld1}
A nodal $MW_\lambda$-link $L_0\subset\rp^3$
is a nodal Hopf link if and only if
the degree of each planar loop is one.
\end{prop}

The following proposition is an immediate
consequence of Proposition \ref{prop-mdeg}.
\begin{prop}\label{prop-degnHl}
Any nodal $MW_\lambda$-link
(in particular, a smooth $MW_\lambda$-link)
can be degenerated
to a nodal Hopf link.
\end{prop}

\ignore{
Suppose $A\subset\cp^3$ is a real algebraic curve of degree $d$ whose
normalization is $\nu:\tilde A\to A$ is at most 2-1 map.
Let $N=\{x\in A\ | \#(\nu^{-1}(x))>1\}\subset A$ and
$\tilde N=\nu^{-1}(N)\subset\tilde A$.
Suppose that $\\tilde A$ is an $M$-curve, and that $N\subset \R A$.
Attaching chords along $\nu^{-1}(x)$, $x\in N$,
produces a real chord diagram $\R X_{\tilde A}$.
\begin{lem}
If $\R X_{\tilde A}$ is planar, and $\#(N)=d-g$.
\end{lem}
}

\begin{lem}\label{lem-SO}
Let $A$ be a real algebraic curve in $\pp^3$.
Let $P$ be a real plane
such that $A\cap P$ is finite.
Let $F\subset A\cap P$ be a finite set and let
$\lambda_1,\dots,\lambda_s$ be loops contained in $\R A\setminus F$
such that all pairwise intersections $\lambda_i\cap\lambda_j$ are finite,
and each loop $\lambda_i$ is non-trivial in $H_1(\rp^3)$. Then
$$
    \deg A \ge s + \sum_{x\in F} (A\cdot P)_x,
$$
where $(A\cdot P)_x$ is the local intersection
number of $A$ and $P$ at $x$.
\end{lem}
\begin{proof}
  It is enough to consider a perturbation $P'$ of $P$
in the class of $\conj$-invariant smooth 4-dimensional 
submanifolds of $\pp^3$.
 We may choose $P'$ to coincide with $P$ near $F$.
In addition we may ensure that $P'\cap\rp^3$
is transverse to each $\lambda_j$
 and all local intersections of $P'$ with $A$ are positive.
  \end{proof}
\ignore{
\begin{proof}
For each $y\in G=(A\cap P)\setminus F$,
we choose neighborhoods
$y\in U_y\subset V_y$ disjoint from $F$.
Let $P'$ be a smooth non-algebraic perturbation of $P$
such that:
\begin{itemize}
\item
$P'$ is conj-invariant, in particular $P'\cap \rp^3$
is a perturbation of $\R P$;
\item
$P'$ coincides with $P$ outside $\bigcup_{y\in Y} V_y$;
\item
For each $y\in Y$, we have:
$P'\cap U_y=P_y\cap U_y$ for a real plane $P_y$ transverse to $A$, and
$P'\cap(V_y\setminus U_y)$ is disjoint from $A$.
\end{itemize}
Then $\deg A = \sum_{z\in P'\cap A}(P\cdot A)_z$
(by the intersection theory in $\cp^3$),
and each $\lambda_i$ contributes at least $1$ to this sum
(by the intersection theory mod $2$ in $\rp^3$).
\qed 
}

The following proposition generalizes Proposition
\ref{prop-iHl}. We use the notations of Section \ref{nRS}.
Let $V_D^\vee$ be the dual space of $V_D$ and $|D|_\Sigma^\vee$ be its projectivization.

Let $X_S$ be a planar nodal real algebraic chord diagram 
in $l=g+1$ circles and 
$\delta$ chords.
\begin{defn}
A {\em Hopf divisor} in $X_S$
is a real divisor  $D\subset \Sigma^*$
of degree $l+\delta+2$
such that the part of $D$ contained
in every planar loop of $X_S$ has odd degree.
\end{defn}
In other words, $D$ has a single point at each of
$l+\delta$ 
planar loop of $X_S$, and also two more
points which either form a complex-conjugate pair,
or both belong to the same planar loop of $X_S$.
Proposition \ref{nHld1} assures that a real
plane section
of a nodal Hopf link is a Hopf divisor.
The next proposition assures the converse.

Recall that a real curve $A_0\subset\pp^3$
is said to be {\em hyperbolic}
with respect to a line $\ell\subset\rp^3$ if
for any real plane $P\subset\pp^3$ passing through $\ell$
each intersection point of $P\setminus\ell$ and $A_0$
is real.
\begin{prop}\label{prop-thmH}
For a Hopf divisor $D\subset\Sigma^*$ the map
\[
\iota_D:S\to |D|_\Sigma^\vee\approx\pp^3,
\]
from \eqref{n-iota}
is a well-defined immersion
whose image is a nodal Hopf link.
Furthermore,
$\iota_D$ factors through an embedding of $\Sigma$ to
$|D|_\Sigma^\vee$. 
\ignore{
Then
there exists
a nodal Hopf link
(of degree $d=l+\delta+2$ and genus $g=l-1$)
whose chord diagram is $(X,\Lambda)$.
Furthermore, all $MW_\lambda$-links of
the same degree and the
same chirality
corresponding to the
same diagram $X$ are rigidly isotopic.
}
\end{prop}
\begin{proof}
\ \\
{\bf Step 1:}
{\it $\iota_D$ is an immersion to $\pp^3$.}
We have $|D|_\Sigma\approx\pp^3$ by Corollary \ref{nsD}. 
The same corollary implies that $\dim |D|_\Sigma >  |D-x|_\Sigma$
for any $x\in \supp D\cap\Sigma^*$.
Also we have
$$\dim |D-\dd I|_{\Sigma^I}= \dim
|D|_{\Sigma^I}-2<|D|_\Sigma.$$
Here the equality follows, once again, from Corollary \ref{nsD}
(recall that $X_S$ is planar, so
both points of $\dd I$ belong to the same 
component of $\R S$) while the inequality 
follows from the observation that $|D|_\Sigma$
has only one additional linear constraint
in addition to those of $|D|_{\Sigma^I}$.
Thus $\iota_D$ is well-defined,
and $A=\iota_D(S)\subset\pp^3$ is a real algebraic curve of degree 
$l+\delta+2$ with plane sections parameterized by $|D|_\Sigma$.

If $\iota_D$ is not an immersion at $x\in S$ then any plane section
containing $x$ has multiplicity greater than one at $x$. Taking
the plane through $x$ and a pair of complex conjugate points of $A$
we get a contradiction with Lemma \ref{lem-SO}.

{\bf Step 2:} {\it $A$ is a nodal curve with exactly $\delta$ nodes.}
If $x\in \R A$ is such that
$\iota_D^{-1}(x)\setminus\R S\neq\emptyset$ then
the plane through $x$ and a pair of conjugate
points of $A$ provides a contradiction with
Lemma \ref{lem-SO}. 
If $x\in A\setminus\rp^3$ is such that $\#(\iota_D^{-1}(x))>1$
then we get a similar contradiction by considering a real plane
containing $x$ and $\conj(x)$. This means that 
$\iota_D|_{S\setminus\R S}$
is an embedding. 

Suppose that we have $\iota_D(x)=\iota_D(y)$
for $x\neq y\in \R S$.
Consider the plane $H$ through $\iota_D(x)$
tangent to the local
branches of $S$ at $x$ and $y$.
The corresponding divisor
has multiplicity at least 2 at $x$ and $y$.
If $x,y\in \R S\setminus\tilde N$
then by the parity reasons there must be 
another intersection point of this plane with
the planar cycles containing $x$ or $y$
which gives a contradiction.
Suppose that $y\in \tilde N$, and $y'\neq x$
is the other endpoint of the same chord.
Then the divisor cut by $H$ 
also contains $y'$.
Let us perturb $H$ to a generic plane section with
two imaginary points near $x$. By the degree count
no other imaginary points may appear under this perturbation. 
Thus $y$ and $y'$ produce at least three real intersection
points which
should be repartitioned among the two planar cycles adjacent to $I$.
Thus at least one planar cycle of $\R X_S$ contains more than
a single point of the divisor.
Since the divisor has a pair of complex conjugate points,
we have a contradiction once again.
Thus $\iota_D$ identifies only the endpoints of the same
chord.

To prove that $A$ is nodal it remains to show that the two branches
at each of its singular points are not tangent to the same direction.
But if they were, we could find a plane tangent to one of the branch
with order at least 3 and, simultaneously, tangent to the other branch
 with order of at least 2, and obtain a contradiction
with Lemma \ref{lem-SO}.

{\bf Step 3:} {\it $A$ is hyperbolic with respect to any real tangent line (for definition, see just before Proposition \ref{prop-thmH})}. 
Using Lemma \ref{lem-SO} we conclude that a plane tangent to $\R A$
has only real intersections with $\R A$ (cf. the condition $(ii)$
in Theorem 2 of \cite{MO-mw}), and also that no plane may be
tangent to a local branch of $\R A$ with order greater than 3.
The latter condition implies that the (differential geometric)
torsion of $\R A$ cannot change sign within the same real branch
of $A$.
Reflecting the orientation of $\rp^3$ if needed, we may assume
that $\R A$ has points of positive torsion.

{\bf Step 4:} {\it Positivity of the linking number
of a pair of oriented tangent lines to $\R A$.}
Here we assume the orientation to be compatible
with a fixed complex orientation of $\R A$.
Due to the hyperbolicity of $\R A$ with respect to tangent lines,
this fact follows from  \cite[Lemma 4.6 and Lemma 4.7]{MO-mw}
(see, in particular, \cite[Figure 2]{MO-mw}).
In its turn this positivity implies that
all points of $\R A$ have positive torsion,
and also that all crossing points
(in the knot-theoretic sense) of a plane projection
of $\R A\subset\rp^3$ are positive.

{\bf Step 5:} {\it There exists a point $p\in\rp^3$
such that all singularities of $\pi_p(A)$
are ordinary real nodes with real local branches.} 
The central projection $\pi_p:A\to \pp^2$
from 
a generic real point $p$ on a tangent line $\ell$ to $\R A$ at 
its non-nodal point has a cusp corresponding to $\ell$.
Furthermore, all singular points of the curve $\pi_p(\R A)\subset\rp^2$
must be real singularities such that all of its branches
are also real. Indeed,
the curve $\pi_p(\R A)$ may not have a pair of complex conjugate
singularities, as otherwise the inverse image under $\pi_p$
of the real line through this pair would intersect $A$ at least in
4 imaginary points producing a contradiction with
Lemma \ref{lem-SO}.
Also the curve $\pi_p(\R A)$ may not have a real singularity $q$
with imaginary branches as otherwise we would get the contradiction
to Lemma \ref{lem-SO}
by considering  the plane passing through $\ell$
and $\pi_p^{-1}(q)$.

Consider the central projection $\pi_{p'}:A\to \pp^2$
from a generic real point $p'$ close to $p$.
The image $\pi_{p'}(A)$ is a real nodal curves with
the nodes of three types: the nodes of $A$, the perturbations
of the nodes of $\pi_p(A)$ and the node resulting from the
cusp of $\pi_p(A)$.
Choosing $p'$ in an approproate way we ensure that
the last node has real local branches.
Recall that all the nodes of $\pi_p(\R A)$ are positive
as knot-theoretical crossing points by the previous step.

{\bf Step 6:} {\it $A$ is a nodal Hopf link.} 
Lemma \ref{positive-perturb} allows us 
to perturb all the nodes of $\R A$ 
in a positive way.
The result of perturbation
has $N_d-g$ positive crossings,
and thus is
a $MW_\lambda$-link.
\ignore{
We claim that $\dim \Delta= 4$, so that $\pp(\Delta^*)\approx\pp^3$.
Otherwise considering a projection 
$\pp(\Delta^*)\dashrightarrow \pp^3$ from an appropriate
linear space not intersecting the image of $\tilde A$
we may find an immersion $\tilde A\to\pp^3$
such that it has more than $\delta$ nodes.
Indeed, a generic projection to $\pp^3$ 
has only the $\delta$ nodes corresponding
to the chords of $(X,\Lambda)$,
and
is not contained in any plane.
A further projection to $\pp^2$
gives us a nodal
curve $B\subset\pp^2$ of geometric genus $g$.
It has all $\delta$ nodes corresponding to 
$(X,\Lambda)$, but also (since it is a projection
of a $\delta$-nodal spatial curve in $\pp^3$)
at least one extra node.
Furthermore, projections from a generic
line in $\pp^4$ give us a pencil of lifts $B\to\pp^3$,
and, in particular, a nodal $MW_\lambda$-link
with more than $\delta=d-3-g$ nodes.
We get a contradiction with Corollary \ref{maxco}. 

Thus the image of the map $\tilde A\to\pp(\Delta^*)\approx\pp^3$ 
is a nodal Hopf link well-determined up to a projective transformation.
Nodal Hopf links with the same chord diagram and the sign of
$w_\lambda$ are rigidly isotopic since 
orientation-preserving projective transformation are
isotopic while orientation-reversing projective transformations reverse
the sign of $w_\lambda$.
}
\end{proof}
\begin{prop}\label{nHlGd}
All nodal Hopf links of the same chirality
(the sign of $w_\lambda$ for
a smooth $MW_\lambda$-perturbation),
and with homeomorphic
chord diagrams are rigidly isotopic,
i.e. isotopic in the class of real nodal algebraic links of the same degree.
\end{prop}
\begin{proof}
A nodal Hopf link is determined by a
Hopf divisor on a planar real algebraic chord diagram
by 
Proposition \ref{prop-thmH}. 
If a homeomorphism between the chord diagrams
of two nodal Hopf links respects the corresponding
Hopf divisors then there exists a 1-parametric
family of Hopf divisors on planar real algebraic 
chord diagrams producing an isotopy between the two
nodal Hopf link.
Lemmas \ref{dmove} and \ref{dmoves} reduce 
the general case to the case considered above.
\end{proof}

Let $(X,\Lambda)$ be a planar chord diagram
in $l$ circles with $\delta$ chords.
We say that a triple $(X,\Lambda,\Delta)$ is
a {\em Hopf triple} if $\Delta\subset \Lambda$ is a set of
$\delta+l$ points disjoint from the
chord endpoints, and such that  
each planar loop of $(X,\Lambda)$
contains a single point of $\Delta$.

Let $I\subset X$ be one of the chords of $(X,\Lambda)$.
The chord $I$ is adjacent to 
two plane loops, $\alpha^I_1$ and $\alpha^I_2$.
Let $\Delta'\subset X$ be a set of $l+\delta-2$
points disjoint from the chord endpoints,
and such that each planar loop of $(X,\Lambda)$
except for $\alpha^I_1$ and $\alpha^I_2$ contains
a single point of $\Delta'$.

Let $\Delta_+$ (resp. $\Delta_-$) be the union of
$\Delta'$
and the two-point set obtained by moving both points
of $\dd I$ in the direction (resp. contrary to
the direction) of the orientation of $\Lambda$. 
Note that both $(X,\Lambda,\Delta_+)$ and
$(X,\Lambda,\Delta_-)$ are Hopf triples.
In this case we say that these triples
are linked with a {\em chord move} in $I$.

\begin{lem}\label{dmove}
For any two Hopf triples $(X,\Lambda,\Delta_\pm)$
linked with a chord move there exists
a nodal Hopf link $A_0\subset\pp^3$
and two generic real planes $H_\pm\in\pp^3$
such that $(\R X_{A_0},\R S, S\cap\R H_{\pm})$
is homeomorphic to $(X,\Lambda,\Delta_\pm)$.
Here $X_{A_0}$ is the natural chord diagram
on the normalization $S$
of the nodal curve $A_0$.
\end{lem}
\begin{proof}
Choose a real algebraic realization
$(\R X_S,\R S,D')$ of $(X,\Lambda,\Delta')$.
Define $D$ to be the divisor obtained
as the union of $D'\cup\dd I$ with
a pair $\Pi$ of complex 
conjugate points in $S\setminus \R S$.
Note that according to Lemma \ref{nsD}
the divisor $D$ is non-special in 
the singular surface $\Sigma^I$ corresponding
to 
the removal of $I$
from the chord diagram $X_S$. 
Thus there exists a real divisor $E$ close to $D$
and $\Sigma^I$-equivalent
to $D$ such that $E\cap\dd I=\emptyset$ and
$(\R X_S,\R S,E\cap\R S)$ is a Hopf triple.
Note that $D\in |E|_\Sigma$
since $D\in |E|_{\Sigma^I}$ and the value of 
a meromorphic function $f:S\to\pp^1$ with 
$(f)=D-E$ is zero (and thus is the same)
on both points of $\dd I$.

By Proposition \ref{prop-thmH}
$A_0=\iota_E(S)\subset\pp^3$
is a nodal Hopf link. Since $D\in |E|_\Sigma$
there exists a real plane $H\in\pp^3$ with 
$\iota_E^{-1}(H)=D$.
In particular, $H$ passes through
the node of $A_0$ corresponding to $I$.
Let $H_\pm$ be the real planes
obtained by perturbing $H$ to two different sides
of the node. Since the pencil of real planes
through $\Pi$ defines a totally real map 
$S\setminus\Pi\to\pp^1$ (i.e. the inverse
image of $\rp^1$ coincides with $\R S$),
the triples $(\R X_{A_0},\R S, S\cap\R H_{\pm})$
and $(X,\Lambda,\Delta_\pm)$ are homeomorphic.
\end{proof}

\begin{lem}\label{dmoves}
Any two Hopf triples
$(X,\Lambda,D_1)$ and $(X,\Lambda,D_2)$
on the same planar chord
diagram $(X,\Lambda)$ can be linked 
with a sequence of chord moves. 
\end{lem}
\begin{proof}
Inductively by $\delta$ we may assume that
the lemma holds for all planar chord diagrams
in less than $\delta$ chords.
Since $(X,\Lambda)$ is dual to a tree,
there exists a chord $I$ dual to a leaf edge
of the tree.

This means that there exists a planar loop 
$\alpha$ adjacent to $I$ and not adjacent to
any other chord of $(X,\Lambda)$. We apply
the induction hypothesis to the diagram
obtained by removing $I$ from $(X,\Lambda)$
and the divisors obtained by removing $\alpha\cap D_j$
from $D_j$, $j=1,2$.  
\end{proof}

\begin{coro}\label{allHd}
Two $MW_\lambda$-links of the same degree
and chirality are rigidly isotopic 
if they can be degenerated to the nodal Hopf link with the
same chord diagram.
\end{coro}
\begin{proof}
To deduce this statement from Proposition \ref{nHlGd} 
it suffices to show that there exists an open neighborhood 
of the map $\iota_D:S\to\pp^3$ (with $\delta$ nodes in the image)
in the space of all holomorphic maps $S\to\pp^3$ of the same degree
such that the space inside this neighborhood
formed by the $MW_\lambda$-links is open and connected.
Corollary \ref{nsD} implies such connectedness once we perturb
only one coordinate of $\pp^3$ leaving the other coordinates
(responsible for the planar diagram in $\rp^2$) unchanged.
Combining this with local connectedness of the space
of perturbations of the planar diagram we get the statement.
\end{proof}

\section{Generalization and proof of Theorem \ref{thmain}}
\label{s-l}
We say that a degree-chord diagram $(X,\Lambda)$,
see Definition \ref{def-degchord}, is 
{\em realized} by an $MW_\lambda$-link $L_0\subset\rp^3$ if $X(L_0)$
is homeomorphic to $(X,\Lambda)$ as the chord diagram
and a generic real plane section of $L_0$ containing a pair
of complex conjugate point has $a(\lambda)$ points on each planar
loop $\lambda$ in $X(L_0)$. Recall that $a(\lambda)$ does
not depend on the choice of the plane section.

The following theorem generalizes Theorem  \ref{thmain} to
nodal $MW_\lambda$-links. Thus its proof also provides
a proof of Theorem \ref{thmain}.
\begin{thm}\label{thm-gen}
The nodal $MW_\lambda$-links of degree $d$
and genus $g$ 
are classified up to rigid isotopy 
by the degree-chord diagrams.

Namely, 
any planar degree-chord diagram $(X,\Lambda)$ in $l=g+1$ circles 
with $\delta$ chords and the degree function on the planar
loops satisfying to \eqref{a-comp} is realized
by a nodal $MW_\lambda$-link
of genus $g$ and degree $d$.
Any two nodal
$MW_\lambda$-links with the same degree-chord
diagram and the same chirality
(i.e. the sign of $w_\lambda$ for its 
smooth $MW_\lambda$-perturbation)
are rigidly isotopic.
\end{thm}

Suppose that $(X,\Lambda)$ is a planar
degree-chord diagram of nodal Hopf link (i.e. with the
degree of each planar loop equal to 1).
Consider the diagram $(Y,\Lambda)$ obtained from $(X,\Lambda)$
by removing
some chords. Then $(Y,\Lambda)$ is also a planar chord diagram.
A planar loop $\lambda$ of $Y$ is composed of several planar loops of $X$.
We set the degree of $\lambda$
to be the number of such
planar loops in $X$.
\begin{lem}
The degree-chord diagram $(Y,\Lambda)$ is realizable as the result of perturbation
of a nodal Hopf link with the chord diagram $(X,\Lambda)$.
\end{lem}
\begin{proof}
Consider a real algebraic chord diagram $X_S$ corresponding
to $(X,\Lambda)$, and a divisor $D$ with a pair of complex conjugate 
points and a single non-nodal point at each of the planar loops of $X_S$.
Let $\iota_D:S\to |D|_\Sigma^\vee\approx\pp^3$
be the corresponding map
(as in \eqref{n-iota}).
As in the proof of Proposition \ref{prop-thmH},
by Corollary \ref{nsD}, Lemma \ref{lem-DI},
and Lemma \ref{positive-perturb}
the nodes corresponding to the chords missing in $(Y,\Lambda)$
may be perturbed in a positive way keeping the other nodes.
The result is a nodal $MW_\lambda$-link $L_Y$ as it could be further
perturbed to a smooth $MW_\lambda$-link.
The number of points of $L_Y$ in its plane section near $D$
agrees with our definition of $a(\lambda)$.
\end{proof}

Suppose $(X,\Lambda)$ is a planar chord diagram,
and $x,y,z\in \Lambda$ be three distinct points different
from the endpoints of the chords and contained in a single
planar loop $\lambda$ of $X$ in the order agreeing with 
the cyclic orientation of $\lambda$.
Consider a point $y'$ on $\lambda$ close to $y$ and a point
$z'$ on $\lambda$ close to $z$ so that the cyclic order
of $x,y,y',z,z'$ agrees with that on $\lambda$.

Form a planar diagram $(Y,\Lambda)$ by attaching 
two new chords to $X$: the one connecting $x$ and $y$ and the
one connecting $y'$ and $z$, see Figure \ref{perestroika}.
\begin{figure}[h]
\includegraphics[height=25mm]{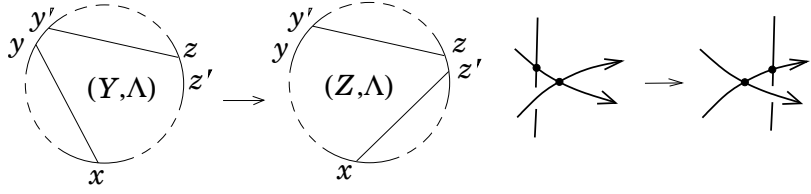}
\caption{The chord diagrams $(Y,\Lambda)$, $(Z,\Lambda)$ and
the resulting modification of nodal links.
\label{perestroika}}
\end{figure}
Also form a planar diagram $(Z,\Lambda)$ by attaching 
two new chords to $X$: the one connecting $x$ and $z'$ and the
one connecting $y$ and $z$.
%
In this case we say that the chord diagrams $(Y,\Lambda)$
and $(Z,\Lambda)$ differ by a {\em chord slide}.

Let $(X',\Lambda)$ be the plane diagram obtained from
$(Y,\Lambda)$ by removing the chord $[xy]$ from $(Y,\Lambda)$ 
(which results in the same diagram as removing
the chord $[xz']$ from $(Z,\Lambda)$).
One of the planar loop of $X'$
corresponds to two planar loops of $Y$ (or of $Z$).
We set its degree equal to 2. All other planar loops
have degree 1. 
This turns 
$(X',\Lambda)$ into a degree-chord diagram.
We consider the diagrams
$(Y,\Lambda)$ and $(Z,\Lambda)$
as the degree-chord diagrams with the degrees
of all planar loops equal to 1.
\begin{lem}\label{lem-slide}
Suppose that two nodal $MW_\lambda$-links $L_y,L_z\subset\rp^3$
both have the degree-chord diagram $(X',\Lambda)$,
and degenerate to nodal
Hopf links with the chord diagrams $(Y,\Lambda)$ and
$(Z,\Lambda)$.
Then $L_y$ and $L_z$ are rigidly isotopic.
\end{lem}
\begin{proof}
Choose a real Riemann surface $S$
of genus $g$ with
$l=g+1$ real branches,
and an orientation-preserving homeomorphism between
$\R S$ and $\Lambda$.
Consider the real algebraic chord diagram $Y^0_S$ on $S$
obtained by attaching all chords of the diagram $(X,\Lambda)$,
a chord connecting $x$ and $y$, and a chord connecting $y$ and $z$.
Also consider the real algebraic chord diagram $Z^0_S$ on $S$
obtained by attaching all chords of the diagram $(X,\Lambda)$,
a chord connecting $x$ and $z$, and a chord connecting $y$ and $z$.
Note that both $Y^0_S$ and $Z^0_S$ are planar diagrams
which are not nodal. Also note that
that $Y^0_S$ and $Z^0_S$ can be perturbed
to nodal diagrams $Y_S$ and $Z_S$ from Figure
\ref{perestroika}.

Choose a divisor on $D$ of degree $l+\delta+2$
consisting of a pair of complex conjugate
points on $S$ and a single non-nodal point at each
planar cycle of $\R Y^0_S$.
Note that then the same divisor will have a single point 
at each planar cycle of $\R Z^0_S$.

Real algebraic chord diagrams $Y^0_S$ and $Z^0_S$ define the same
map $\iota_D:S\to |D|_\Sigma\approx\pp^3$
since the linear system $|D|_\Sigma$
is the same for both diagrams: the functions in $V_D$ 
have the same values at $x$, $y$ and $z$.
The image $A=\iota_D(S)$ has a triple point with
three real branches
at $\tau=\iota_D(x)=\iota_D(y)=\iota_D(z)$,
the rest of the singularities are 
nodes corresponding to the chords of $(X,\Lambda)$.
This can be seen in the same way as in the proof
of Proposition \ref{prop-thmH}.

We claim that all three local branches of $A$ cannot be tangent
to the same plane, as otherwise the corresponding
plane section has a tangency
of order 6 that must be repartitioned to the planar
cycles.
There are $l+g-3$ planar cycles which are not 
adjacent to the points $x,y,z\in Y^0_S$. 
As the degree of each of them is 1,
and the degree
of $A$ is $l+g+2$ we get a contradiction.
Thus we may choose a generic point $p\in\rp^3$
so that the local intersection sign of the $y$-branch and
the $z$-branch with the $x$-branch
of $\pi_p\circ\iota_D(\R S)\subset\rp^2$ coincide,
see Figure \ref{xyz-sign-perturb}.
\begin{figure}[h]
\includegraphics[height=15mm]{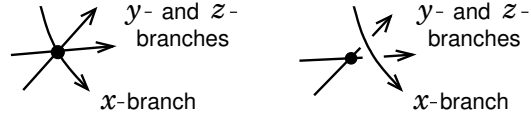}
\caption{Perturbation of the triple point.
\label{xyz-sign-perturb}}
\end{figure}
As in the proof of Proposition \ref{prop-thmH} we change the
coordinates in $\rp^3$ so that $p=(0:0:0:1)$. 

Let $\Sigma'$ be the singular surface obtained
by contracting $S$ along each chord of $(X',\Lambda)$. 
Using Corollary \ref{nsD} and Lemma \ref{lem-DI}
(cf. Lemma \ref{positive-perturb})
we may perturb the map $\iota_D$
keeping the image of the curve
under $\pi_p$ (the plane diagram)
so that it factors through an embedding of $\Sigma'$,
and so that the new crossing points resulting
from the perturbation are positive, 
see Figure \ref{xyz-sign-perturb}.
Performing the same perturbation for
a family of chord diagrams connecting 
$Y^0_S$ to $Y_S$ and $Z^0_S$ to $Z_S$
we obtain an isotopy between $MW_\lambda$-links
degenerating to the nodal Hopf links with
the chord diagrams $(Y,\Lambda)$ and $(Z,\Lambda)$.
\end{proof}

\begin{proof}[Proof of Theorem \ref{thm-gen}]

If we have two nodal $MW_\lambda$-links corresponding to
the same planar degree-chord diagram $(X,\Lambda)$ then both of them
degenerate to
nodal Hopf links as in Proposition \ref{prop-mdeg}.
The chord diagrams 
of these nodal Hopf links must contain $(X,\Lambda)$
as a subdiagram so that the degree of a planar cycle in $X$
equals to the number of the planar cycles that appear in
the subdivisions of this cycle in the larger diagrams.
But then the two larger chord diagrams can be connected
with a sequence of chord slides. Theorem now follows
from Lemma \ref{lem-slide}.
\end{proof}

\bibliography{b}
\bibliographystyle{plain}

\end{document}